\documentclass[a4paper]{article}
\usepackage{amsmath,amssymb}
\usepackage[utf8]{inputenc}
\usepackage{tikz}
\usepackage[english, frenchb]{babel}
\usepackage[french,colorlinks,citecolor=blue]{hyperref}
\usetikzlibrary{arrows,trees}

\newtheorem{theorem}{Théorème}[section]

\newtheorem{lemma}[theorem]{Lemme}

\numberwithin{equation}{section}

\newcommand{\ZZ}{\mathbb{Z}}
\newcommand{\NN}{\mathbb{N}}
\newcommand{\QQ}{\mathbb{Q}}

\newcommand{\qbinom}[2]{\genfrac{[}{]}{0mm}{1}{#1}{#2}}
\newcommand{\qbase}[3]{\genfrac{[}{]}{0mm}{1}{#1,\,#2}{#3}}

\newcommand{\asc}{\operatorname{Asc}}
\newcommand{\desc}{\operatorname{Desc}}

\newcommand*\pFq[4]{{}_{#1}\phi_{#2}\biggl(\genfrac..{0pt}{}{#3}{#4};q,q\biggr)}

\newenvironment{proof}{\begin{trivlist}\item{\bf{Preuve.}}}
  {\hfill\rule{2mm}{2mm}\end{trivlist}}

\title{Nombres de $q$-Bernoulli-Carlitz et fractions continues}
\author{F. Chapoton et J. Zeng}
\date{\today}
\begin{document}

\maketitle
\selectlanguage{english}
\begin{abstract}
Carlitz has introduced $q$-analogues of the Bernoulli numbers around
1950. We obtain a representation of these $q$-Bernoulli numbers (and
some shifted version) as moments of some orthogonal polynomials. This
also gives factorisations of Hankel determinants of $q$-Bernoulli
numbers, and continued fractions for their generating series. Some of
these results are $q$-analogues of known results for Bernoulli numbers,
but some are specific to the $q$-Bernoulli setting.
\end{abstract}
\selectlanguage{frenchb}
\begin{abstract}
Carlitz a introduit vers 1950 des $q$-analogues des nombres de
Bernoulli. On obtient une représentation de ces $q$-analogues (ainsi
que de variantes décalées) comme moments de certains polynomes
orthogonaux. Ceci donne aussi des factorisations des déterminants de
Hankel des nombres de $q$-Bernoulli, ainsi que des fractions continues
pour leurs séries génératrices. Certains de ces résultats sont des $q$-analogues d'énoncés
connus pour les nombres de Bernoulli, mais d'autres sont sans version classique.
\end{abstract}

\section*{Introduction}

Les nombres de Bernoulli ont une longue histoire, et sont utiles dans
divers domaines des mathématiques. Leur apparition dans les valeurs
prises aux entiers par la fonction $\zeta$ de Riemann est sans doute
une des principales raisons de leur importance.

Un aspect apparemment assez peu connu est l'existence de plusieurs
fractions continues simples pour des séries génératrices faisant
intervenir des nombres de Bernoulli ou des polynômes de Bernoulli. On
trouve déjà un exemple de ce type pour les polynômes de Bernoulli dans
les travaux fondateurs de Stieltjes sur les fractions continues
\cite[\S 86]{stieltjes_toulouse2}. D'autres exemples pour les nombres
de Bernoulli sont présentés dans l'appendice par Zagier du livre
\cite{zagier}. Sans aucune exhaustivité, on trouve notamment de telles
fractions continues dans les articles \cite[\S 6]{rogers}, \cite[\S
5]{frame_hankel} et \cite[\S 13 et 14]{touchard}.

Le contexte naturel pour ces fractions continues (du moins pour celles
qui font intervenir des séries génératrices ordinaires) est la théorie
des polynômes orthogonaux. Les nombres ou polynômes de Bernoulli
apparaissent dans ce cadre comme les moments de familles de polynômes
orthogonaux en une variable. On peut citer notamment Carlitz
\cite{carlitz_ortho} pour une interprétation en ces termes des
résultats de Stieltjes pour les polynômes de Bernoulli. Ces travaux
sont brièvement décrits dans \cite[p. 191-192]{chihara}.

Une conséquence de cette réalisation comme moments de polynômes
orthogonaux est l'existence de formules closes pour les déterminants
de Hankel des nombres de Bernoulli. En particulier, de telles formules
sont démontrées dans \cite{alsalam_carlitz, kratt_det} pour les
déterminants de Hankel des nombres de Bernoulli avec indices décalés
de $0,1$ ou $2$. Un résultat plus général, qui contient ces trois cas, a
également été obtenu par Fulmek et Krattenthaler dans
\cite{kratt_det}.

Notre objectif dans cet article est d'obtenir des résultats similaires
(fractions continues pour les séries génératrices, factorisations des
déterminants de Hankel, expressions comme moments de polynômes
orthogonaux) pour les $q$-analogues des nombres de Bernoulli et des
polynômes de Bernoulli introduits par Carlitz dans
\cite{carlitz_qbern}.

Ces nombres de $q$-Bernoulli-Carlitz, qui sont des fractions
rationnelles en la variable $q$, semblent assez naturels. Ils sont
apparus récemment dans l'étude de certaines séries formelles en
arbres \cite{chap-tree}, ainsi que dans la théorie des $q$-polynômes d'Ehrhart \cite{chap-essouabri}, et sont
fortement liés avec un $q$-analogue de la fonction $\zeta$ \cite{chap-zeta}. Les
résultats du présent article sont un autre signe de leur intérêt potentiel.

L'article est organisé comme suit. Après un section d'introduction
des notations, on obtient d'abord des expressions des nombres de
$q$-Bernoulli-Carlitz (éventuellement décalés) comme moments de
polynômes orthogonaux de type $q$-Hahn ou $q$-Legendre. On formule
ensuite les récurrences explicites pour ces polynômes orthogonaux. Par
la théorie générale, ceci donne des fractions continues de Jacobi pour
les séries génératrices des moments et des factorisations de
déterminants de Hankel des moments. On transforme ensuite ces fractions
continues en fractions continues de Stieltjes. Enfin, on obtient dans
la dernière section un résultat plus général, qui est un $q$-analogue
du théorème de Fulmek et Krattenthaler. Ceci fait intervenir les
polynômes orthogonaux de type $q$-Jacobi.

Dans la plupart de nos résultats, on peut faire $q=1$ et retrouver de
manière transparente les résultats classiques. Il y a deux exceptions:
les fractions continues pour la série génératrice décalée de deux
crans, et la formule pour le déterminant de Hankel décalé de trois
crans, qui n'ont pas de limite classique et sont donc des formules
complètement nouvelles.

Le lecteur intéressé pourra trouver beaucoup d'informations
historiques sur le cas classique dans l'introduction de l'article
\cite{koelink}. Un $q$-analogue de la série génératrice exponentielle
des nombres de Bernoulli a été étudié sous un angle similaire dans
\cite{andrews}.

Remerciements: les auteurs remercient le NIMS (Daejeon) où cet article
a été terminé, pour l'accueil et les agréables conditions de travail. Le
premier auteur bénéficie du soutien du contrat ANR CARMA (ANR-12-BS01-0017).

\section{Notations}

Les nombres de $q$-Bernoulli-Carlitz, notés $\beta_n$ pour
$n\geq 0$, sont définis par les relations
\begin{equation}
  \label{defbeta}
  q (q \beta + 1)^n - \beta^n = \begin{cases}
q-1 \quad &\text{si}\,\,n=0,\\
1 \quad &\text{si}\,\,n=1,\\
0 \quad &\text{si}\,\,n>1,
\end{cases}
\end{equation}
où on convient de transformer les exposants de $\beta$ en indices
après avoir développé le binôme.

Ce sont des fractions rationnelles en la variable $q$, dont les valeurs
en $q=1$ sont les nombres de Bernoulli usuels. Pour illustration, voici les premières fractions:

\begin{align*}
\beta_0 &= 1,\quad
\beta_1 = \frac{-1}{q + 1},\quad
\beta_2 = \frac{q}{(q + 1) \cdot (q^{2} + q + 1)},\\
\beta_3 &= \frac{-q \cdot (q-1)}{(q + 1) \cdot (q^{2} + q + 1) \cdot (q^{2} + 1) },\\
\beta_4 &= \frac{q \cdot (q^{4} - q^{3} - 2q^{2} - q + 1)}{(q + 1) \cdot (q^{2} + q + 1) \cdot (q^{2} + 1) \cdot (q^{4} + q^{3} + q^{2} + q + 1) }.
\end{align*}

Leur série génératrice exponentielle est notée
\begin{equation}
  B(x) = \sum_{n \geq 0} \beta_n \frac{x^n}{n!}.
\end{equation}
Des relations \eqref{defbeta}, on déduit qu'elle satisfait l'équation fonctionnelle
\begin{equation}
  q e^x B(q x) - B(x) = q - 1 + x.
\end{equation}

Leur série génératrice ordinaire est notée
\begin{equation}
  \widehat{B}(x) = \sum_{n \geq 0} \beta_n x^n.
\end{equation}
On déduit des relations \eqref{defbeta} l'équation fonctionnelle
\begin{equation}
  \frac{q}{1-x}\widehat{B}\left(\frac{q x}{1-x}\right) - \widehat{B}(x) = q - 1 + x.
\end{equation}

On considère aussi les séries génératrices ordinaires décalées de un
ou deux crans définies par
\begin{equation}
  \widehat{B}_1(x) = \frac{1}{\beta_1}\sum_{n \geq 0} \beta_{n+1} x^n
\quad\text{et}\quad
  \widehat{B}_2(x) = \frac{1}{\beta_2}\sum_{n \geq 0} \beta_{n+2} x^n.
\end{equation}

On note $\Psi$ la forme linéaire sur les polynômes en une variable
$x$ à coefficients dans $\QQ(q)$ définie par
\begin{equation}
  \Psi(x^n) = \beta_n
\end{equation}
pour $n \geq 0$.

Les polynômes de $q$-Bernoulli-Carlitz sont les polynômes en une variable
$z$ à coefficients dans $\QQ(q)$ définis pour $n \geq 0$ par
\begin{equation}
  \label{poly_qbc}
  \beta_n(z) = \Psi((z + (z(q-1)+1)x)^n),
\end{equation}
et leur valeur en $z=0$ est le nombre de $q$-Bernoulli-Carlitz
$\beta_n$. Ce sont des $q$-analogues des polynômes de Bernoulli
usuels. Les trois premiers sont
\begin{equation*}
  1, \frac{2z-1}{q+1}, \frac{3 (q + 1) z^2 - 2(2 q + 1) z + q}{(q + 1)(q^2 + q + 1)}.
\end{equation*}
Ces polynômes de $q$-Bernoulli-Carlitz sont reliés aux polynômes
introduits initialement par Carlitz dans \cite{carlitz_qbern} par un
changement de variable simple. Plus précisément, si on note
$\hat{\beta}_n$ les polynômes originaux de Carlitz, il résulte de la
comparaison entre la formule (5.5) de cette référence et
\eqref{poly_qbc} que
\begin{equation}
  \label{q_poly_relation}
  \beta_n\left((q^y-1)/(q-1)\right) = \widetilde{\beta}_n(y).
\end{equation}

Pour $n\in \ZZ$, on note $[n]_q$ le $q$-entier $(q^n-1)/(q-1)$. Pour
$n\in \NN$, on note $[n]!_q$ la $q$-factorielle $[1]_q [2]_q \dots
[n]_q$. Pour $0 \leq m \leq n$ dans $\NN$, on note $\qbinom{n}{m}_q$ le
$q$-analogue habituel des coefficients binomiaux, défini par
\begin{equation}
  \frac{[n]!}{[m]!_q[n-m]!_q}.
\end{equation}

Pour des entiers positifs ou nuls $i, d$, on pose
\begin{equation}
  \label{q_bino_dd}
  \qbase{i}{x}{d}_q = \frac{1}{[d]!_q} ([i-d+1]_q+q^{i-d+1} x)([i-d+2]_q+q^{i-d+2} x)\dots([i]_q+q^i x).
\end{equation}
C'est un $q$-analogue du polynôme binomial $\binom{i+x}{d}$.

On a alors les évaluations suivantes (voir \cite[Prop. 3.3 et 3.5]{chap-essouabri}).
\begin{lemma}
  \label{beta}
  Pour des entiers $0 \leq i \leq d$, on a
  \begin{equation}
    \Psi(\qbase{i}{x}{d}_q) = \frac{(-1)^{d-i}q^{-\binom{d-i}{2}}}{[d+1]_q \qbinom{d}{i}_q}.
  \end{equation}
\end{lemma}
\begin{lemma}
  \label{alpha}
  Pour des entiers $0 \leq i \leq d$ et $0 \leq j \leq e$, on a
  \begin{equation}
    \Psi(\qbase{i}{x}{d}_q\qbase{j}{x}{e}_q) = \frac{(-1)^{d-i+e-j}q^{-\binom{d-i}{2}+(d-i)(e-j)-\binom{e-j}{2}}}{[d+e+1]_q \qbinom{d+e}{d-i+j}_q}.
  \end{equation}
\end{lemma}

On introduit les notations
\begin{equation}
  \label{def_asc}
  \asc(x, 0) = 1, \quad  \asc(x, a) = \prod_{i=1}^{a} ([i]_q + q^i x)
\end{equation}
et
\begin{equation}
  \label{def_desc}
  \desc(x,-1) = -1/x, \quad \desc(x, 0) = 1, \quad  \desc(x, a) = \prod_{i=1}^{a}([i]_q-x).
\end{equation}
Remarque : pour $n$ entier positif, $[-n]_q+q^{-n} x = - q^{-n} ([n]_q - x)$.

On utilise le symbole de Pochhammer de base $q$ défini par
\begin{equation}
  (a ; q)_ k = (1-a) (1-q a)\dots(1-q^{k-1} a),
\end{equation}
ainsi que la notation abrégée $(a,b,\dots;q)_k$ pour le produit de plusieurs tels symboles.

On a besoin du lemme suivant (voir \cite[p. 25]{chihara}).
\begin{lemma}
  \label{affine}
  Soit $(q_n(x))_{n \geq 0}$ une famille de polynômes orthogonaux définis par
  \begin{equation}
    q_{n+1}(x) = (a_n + x) q_{n}(x) - b_n q_{n-1}(x),
  \end{equation}
  avec les conditions initiales $q_{-1}(x)=0$ et $q_{0}(x)=1$. Soit
  $(p_n(x))_{n \geq 0}$ les polynômes définis par le changement de
  variables $p_n(x) = q_n(A x+B)/A^n$ $(A\not= 0)$. Alors les $p_n$
  sont une famille de polynômes orthogonaux vérifiant la récurrence
  \begin{equation}
    p_{n}(x) = ((a_n + B)/A + x) p_{n}(x) - b_n/A^2 p_{n-1}(x).
  \end{equation}
  Si on note $\nu_n$ les moments de la famille $q_n$, alors les
  moments de la famille $p_n$ sont
  \begin{equation}
    \mu_n = A^{-n} \sum_{k=0}^{n} \binom{n}{k} (-B)^{n-k} \nu_k.
  \end{equation}
\end{lemma}

\section{Polynômes orthogonaux et moments}

\subsection{Polynômes orthogonaux de type $q$-Hahn}

On considère les polynômes en $x$ définis par la formule
\begin{equation}
  \label{hypergeo_cd}
  P_n(x) = \pFq{3}{2}{q^{-n}, q^{c + d + n + 1},
    q (1 + (q - 1) x)}{q^{c + 1}, q^{d + 1}}
\end{equation}
où ${}_3\phi_{2}$ est la fonction hypergéométrique basique usuelle. Les
paramètres $c$ et $d$ sont des entiers positifs ou nuls.

Ces polynômes sont l'évaluation en $q (1 + (q - 1) x)$ de polynômes
$Q_n$ de type $q$-Hahn, qui forment une famille classique de polynômes
orthogonaux. Ce sont encore des polynômes orthogonaux (voir le lemme
\ref{affine}).

% Paramètres de Hahn:
% \begin{equation}
%   \alpha = A - 1, \quad \beta = C - A -1, \quad N = B - C.
% \end{equation}

% \begin{equation}
%   c = \alpha \quad d = -N -1
% \end{equation}

% \begin{tabular}{|c|c|c|c|c|c|c|}
%   \hline
%   $A$ & $B$ & $C$ & $\alpha = c$ & $\beta$ & $N$ & d\\
%   \hline
%   1 & 1 & 2 & 0& 0& -1 & 0\\
%   2 & 2 & 3 & 1& 0& -1 & 0\\
%   2 & 2 & 4 & 1& 1& -2 & 1\\
%   \hline
% \end{tabular}

\begin{theorem}
  Les nombres de $q$-Bernoulli-Carlitz $(\beta_n)_{n\geq 0}$ sont les
  moments des polynômes orthogonaux $P_n$ de paramètres $c=d=0$.
  % $(A,B,C)=(1,1,2)$
\end{theorem}

\begin{proof}
  Il suffit de montrer que $\Psi$ s'annule sur les polynômes $P_n$
  pour ces paramètres lorsque $n\geq 1$ et vaut $\beta_0 = 1$ lorsque
  $n=0$. Ceci caractérise l'application moment pour cette famille de
  polynômes, voir par exemple la preuve du théorème de Favard dans
  \cite[p. 21-22]{chihara}.

  L'expression hypergéometrique \eqref{hypergeo_cd} donne la formule explicite
  \begin{equation}
    P_n(x) = \sum_{k=0}^{n} \frac{(q^{-n}, q^{n+1} ; q)_k (q (1 + (q - 1) x);q)_k q^k}{(q,q,q;q)_k}
  \end{equation}
  en termes de symboles de Pochhammer.

  Le symbole $(q (1 + (q - 1) x);q)_k$ vaut
  \begin{equation}
    (1-q)^{k} (1+ q x)([2]_q+q^2 x)([3]_q+q^3 x) \dots ([k]_q+q^k x),
  \end{equation}
  ce qu'on peut encore écrire, avec la notation introduite dans \eqref{q_bino_dd},
  \begin{equation}
    (1-q)^k [k]!_q \qbase{k}{x}{k}_q = (q ; q)_k \qbase{k}{x}{k}_q.
  \end{equation}
  Un cas particulier du lemme \ref{beta} donne que
  \begin{equation}
    \label{eval_psi0}
    \Psi(\qbase{k}{x}{k}_q)= \frac{1}{[k+1]_q}.
  \end{equation}
  Par conséquent,
  \begin{equation}
    \Psi(P_n(x)) = \sum_{k=0}^{n} \frac{(q^{-n}, q^{n+1} ; q)_k q^k}{(q,q;q)_k [k+1]_q} = \sum_{k=0}^{n} \frac{(q^{-n}, q^{n+1} ; q)_k q^k}{(q,q^2;q)_k}.
  \end{equation}
  On reconnaît $\pFq{2}{1}{q^{-n}, q^{n + 1}}{q^{2}}$, qui
  est nul par la formule de $q$-Vandermonde lorsque $n\geq 1$.
\end{proof}

% \begin{proof}
%   PREUVE AVEC STIRLING (vaut mieux eviter)

%   Formule connue (Carlitz):
%   \begin{equation}
%     \beta_n = \sum_{k=0}^{n} (-1)^k \sti{n}{k} \frac{[k]!_q}{[k+1]_q}.
%   \end{equation}
%   avec les $q$-Stirling.

%   Formule connue pour les Stirling
%   \begin{equation}
%     \sum_{n \geq k} \sti{n}{k} t^n = \frac{t^k}{(1-[1]_q t)(1-[2]_q t)\dots(1-[k]_q t)}.
%   \end{equation}

%   On en déduit (calculs chez Zeng)
%   \begin{equation}
%     \sum_{n\geq 0} \beta_n t^n = \sum_{k \geq 0} \frac{[k]!_q}{[k+1]_q}\frac{t^k}{(1-[1]_q t)(1-[2]_q t)\dots(1-[k]_q t)}.
%   \end{equation}

%   On obtient donc l'évaluation
%   \begin{equation}
%     \Psi(({[1]_q}{q^{-1}}+x)({[2]_q}{q^{-2}}+x)\dots({[d]_q}{q^{-d}}+x))
%     =\frac{q^{-\binom{d+1}{2}}[d]!_q}{[d+1]_q}.
%   \end{equation}

%   Changement de base des Stirling
%   \begin{equation}
%     x^n = \sum_k \sti{n}{k}_q x(x-[1]_q)(x-[2]_q)\dots(x-[k-1]_q).
%   \end{equation}

%   On doit aussi avoir
%   \begin{equation}
%     \Psi(x(x-[1]_q)(x-[2]_q)\dots(x-[k-1]_q)) = (-1)^k \frac{[k]!_q}{[k+1]_q}
%   \end{equation}

%   On a donc l'égalité de séries génératrices
%   \begin{equation}
%     \Psi\left(\frac{1}{1-x t}\right) = \sum_{n \geq 0} \frac{q^{-\binom{n+1}{2}}[n]!_q}{[n+1]_q} \frac{t^n}{(1+[1]_q q^{-1}t)(1+[2]_q q^{-2}t))\dots(1+[n+1]_q q^{-n-1}t)}.
%   \end{equation}
%   soit encore
%   \begin{equation}
%     \Psi\left(\frac{1}{1-x t}\right) = \sum_{n \geq 0} \frac{[n]!_q}{[n+1]_q} \frac{q^{n+1} t^n}{(q+[1]_q t)(q^2+[2]_q t))\dots(q^{n+1}+[n+1]_q t)}.
%   \end{equation}

% \end{proof}

\begin{theorem}
  Les fractions $(\beta_n / \beta_1)_{n\geq 1}$ sont les moments des
  polynômes orthogonaux $P_n$ de paramètres $c=0$ et $d=1$.
  % $(A,B,C)=(2,2,3)$
\end{theorem}

\begin{proof}
  Il suffit à nouveau de montrer que la forme linéaire $f \mapsto
  \Psi(x f)$ s'annule sur les polynômes $P_n$ pour ces paramètres lorsque $n\geq 1$ et vaut $\beta_1$ lorsque $n=0$.

  L'expression hypergéometrique \eqref{hypergeo_cd} donne la formule explicite
  \begin{equation}
    P_n(x) = \sum_{k=0}^{n} \frac{(q^{-n}, q^{n+2} ; q)_k (q (1 + (q - 1) x);q)_k q^k}{(q,q,q^2;q)_k}.
  \end{equation}

  L'expression $x (q (1 + (q - 1) x);q)_k$ vaut $[k+1]_q (q ; q)_k \qbase{k}{x}{k + 1}_q$.

  Un cas particulier du lemme \ref{beta} donne que
  \begin{equation}
    \label{eval_psi1}
    \Psi(\qbase{k}{x}{k+1}_q)= \frac{-1}{[k+1]_q[k+2]_q}.
  \end{equation}
  On en déduit alors que
  \begin{equation}
    \Psi(x P_n(x)) = - \sum_{k=0}^{n} \frac{(q^{-n}, q^{n+2} ; q)_k q^k}{(q,q^2;q)_k [k+2]_q} = -\frac{1}{[2]_q} \sum_{k=0}^{n} \frac{(q^{-n}, q^{n+2} ; q)_k q^k}{(q,q^3;q)_k}.
  \end{equation}
  On reconnaît la somme comme $\pFq{2}{1}{q^{-n}, q^{n + 2}}{q^{3}}$, qui
  est nul par la formule de $q$-Vandermonde lorsque $n\geq 1$.
\end{proof}

\begin{theorem}
  Les fractions $(\beta_n / \beta_2)_{n\geq 2}$ sont les moments des
  polynômes orthogonaux $P_n$ de paramètres $c=1$ et $d=1$.
  % $(A,B,C)=(2,2,4)$
\end{theorem}

\begin{proof}
  Il suffit à nouveau de montrer que la forme linéaire $f \mapsto
  \Psi(x^2 f)$ s'annule sur les polynômes $P_n$ pour ces paramètres lorsque $n\geq 1$ et vaut $\beta_2$ lorsque $n=0$.

  L'expression hypergéometrique \eqref{hypergeo_cd} donne la formule explicite
  \begin{equation}
    P_n(x) = \sum_{k=0}^{n} \frac{(q^{-n}, q^{n+3} ; q)_k (q (1 + (q - 1) x);q)_k q^k}{(q,q^2,q^2;q)_k}.
  \end{equation}
  L'expression $x^2 (q (1 + (q - 1) x);q)_k$ vaut $[k+1]_q (q ; q)_k \qbase{0}{x}{1}_q \qbase{k}{x}{k + 1}_q$.
  Un cas particulier du lemme \ref{alpha} donne que
  \begin{equation}
    \label{eval_psi2}
    \Psi(\qbase{0}{x}{1}_q \qbase{k}{x}{k+1}_q)= \frac{q}{[k+2]_q[k+3]_q}.
  \end{equation}
  On en déduit alors que
  \begin{equation}
    \Psi(x^2 P_n(x)) = q \sum_{k=0}^{n} \frac{(q^{-n}, q^{n+3} ; q)_k [k+1]_q q^k}{(q^2,q^2;q)_k [k+2]_q [k+3]_q} = \frac{q}{[2]_q[3]_q} \sum_{k=0}^{n} \frac{(q^{-n}, q^{n+3} ; q)_k q^k}{(q,q^4;q)_k}.
  \end{equation}
  On reconnaît la somme comme $\pFq{2}{1}{q^{-n}, q^{n + 3}}{q^{4}}$, qui
  est nul par la formule de $q$-Vandermonde lorsque $n\geq 1$.
\end{proof}

\subsection{Polynômes orthogonaux de type $q$-Legendre}

On introduit une autre famille de polynômes en $x$ définis par la formule
\begin{equation}
  \label{hypergeo_avec_z}
  P_n(x) = \pFq{3}{2}{q^{-n}, q^{n + 1},
    q (1 + (q - 1) x)} {q, q (1 + (q - 1) z)},
\end{equation}
où le paramètre $z$ est une variable.

Ces polynômes sont l'évaluation en $q (1 + (q - 1) x)$ de polynômes
$Q_n$ de type ``grand $q$-Legendre'', qui forment une famille
classique de polynômes orthogonaux (voir \cite[\S 14.5.1]{koekoek_book}).

Lorsque $z=0$, on retrouve le cas $c=d=0$ de type $q$-Hahn considéré
précédemment.

On va calculer leurs moments en termes de polynômes en $z$ obtenus par
intégration des polynômes de $q$-Bernoulli-Carlitz définis par
\eqref{poly_qbc}.

Pour simplifier les notations, on pose $c = 1+(q-1)z$ dans cette section.

On rappelle que la $q$-intégrale de Jackson est définie par
\begin{equation}
\int_a^b f(t)d_qt=b(1-q)\sum_{k=0}^\infty f(bq^k)q^k-a(1-q)\sum_{k=0}^\infty f(aq^k)q^k.
\end{equation}

\begin{lemma}
  Les moments des polynômes $P_n$ sont donnés par
  \begin{align}
    \mu_n=\frac{1}{(q-1)^n}\sum_{k=0}^n {\binom{n}{k}} (-1)^{n-k} \frac{[k+1]_c}{[k+1]_q}.
  \end{align}
\end{lemma}
\begin{proof}
  On sait (voir \cite{koekoek_book}) que les polynômes ``grand $q$-Legendre''
  $Q_n$ de paramètre $c$ sont orthogonaux pour les moments
  \begin{equation}
    \nu_n=\frac{1}{\int_{cq}^q x^0 d_qx}\int_{cq}^q x^n d_qx=
    \frac{q^{n}(1-c^{n+1})(1-q)}{(1-c)(1-q^{n+1})}=q^n\frac{[n+1]_c}{[n+1]_q}.
  \end{equation}
  
  Cette égalité résulte du calcul suivant :
  \begin{align*}
    \int_{cq}^q x^n d_qx&=q(1-q)\sum_{k\geq 0}(q^{k+1})^n q^k-
    cq(1-q)\sum_{k\geq 0}(cq^{k+1})^n q^k\\
    &=q^{n+1}(1-q)\left(\sum_{k\geq 0}(q^{n+1})^k-c^{n+1} \sum_{k\geq 0}(q^{n+1})^k\right)\\
    &=q^{n+1}(1-q)\frac{1-c^{n+1}}{1-q^{n+1}}.
  \end{align*}
  Par le lemme \ref{affine}, les moments des polynômes $P_n$ sont donc
  \begin{align*}
    \mu_n = q^{-n}(q-1)^{-n} \sum_{k=0}^n {\binom{n}{k}} (-q)^{n-k} \nu_k,
  \end{align*}
  ce qui donne le résultat voulu.
\end{proof}

\begin{lemma}
  On a
  \begin{align}\label{last}
    \frac{1}{z}\int_0^z\beta_n(y) dy=\frac{1}{(q-1)^n}\sum_{k=0}^n {\binom{n}{k}} (-1)^{n-k} \frac{[k+1]_c}{[k+1]_q}.
  \end{align}
\end{lemma}
\begin{proof}
  On note que $1-c=(1-q)z$ et que
  \begin{equation}
    \frac{d}{dz}\left(1-c^{k+1}\right)=(1-q)(k+1)c^k.
  \end{equation}
  En multipliant \eqref{last} par $z$, puis en dérivant par rapport à $z$,
  on obtient
  \begin{align}\label{check}
    \beta_n(z)=\frac{1}{(q-1)^n}\sum_{k=0}^n {\binom{n}{k}} (-1)^{n-k} \frac{(k+1)}{[k+1]_q}(1+(q-1)z)^k.
  \end{align}
  Cette équation est exactement la formule (5.3) de
  \cite{carlitz_qbern}, modulo le changement de variables
  \eqref{q_poly_relation}. Il reste à vérifier que \eqref{last} est
  vraie en $z=0$, ce qui résulte aussi de \eqref{check} en $z=0$.
\end{proof}

\begin{theorem}
  Les polynômes $(\frac{1}{z}\int_{0}^z
  \beta_n(y) dy)_{n\geq 0}$ sont les moments des polynômes orthogonaux
  $P_n$.
\end{theorem}
\begin{proof}
  Ceci résulte des deux lemmes précédents. Il n'est pas nécessaire de
  normaliser les moments, car $\beta_0(z)=1$.
\end{proof}

On introduit la série génératrice des moments
\begin{equation}
  \widehat{B}_z(x)= \sum_{n \geq 0} \left(\frac{1}{z}\int_{0}^{z}\beta_n(y) dy\right) x^n. 
\end{equation}

\section{Récurrences}
\label{recurs}

Chaque famille de polynômes orthogonaux (supposés unitaires) vérifie
une récurrence à trois termes, de la forme
\begin{equation}
  p_{n+1} = (a_n + x) p_n - b_n p_{n-1},
  % CONVENTION de Kratt. dans Advanced Det. Calc.
\end{equation}
pour deux suites de coefficients $a_n$ et $b_n$. On va calculer ces
récurrences pour les familles de polynômes considérées, en partant de
la récurrence générale connue pour les polynômes de type $q$-Hahn et
de type ``grand $q$-Legendre''.

\subsection{Récurrences pour le type $q$-Hahn}

Selon \cite[\S 14.6]{koekoek_book}, la version unitaire $q_n$ des
polynômes $Q_n$ définis par la formule \eqref{hypergeo_cd} (avec $x$
au lieu de $q(1+(q-1)x)$) vérifie la récurrence
\begin{equation}
  q_{n+1} = (A_n + C_n - 1 + x) q_n - A_{n-1} C_n q_{n-1},
\end{equation}
où
\begin{equation}
  A_n = \frac{(1-q^{n+d+1})(1-q^{n+c+1})(1-q^{n+c+d+1})}{(1-q^{2n+c+d+1})(1-q^{2n+c+d+2})}
\end{equation}
et
\begin{equation}
  C_n = -\frac{q^{n+c+d+1}(1-q^n)(1-q^{n+c})(1-q^{n+d})}{(1-q^{2n+c+d})(1-q^{2n+c+d+1})}.
\end{equation}

On utilise ensuite le lemme \ref{affine} pour le changement de
variables $x \mapsto q(1+(q-1)x)$. On obtient la récurrence
\begin{equation}
  p_{n+1} = \left((A_n + C_n - 1 + q)/(q(q-1)) + x\right) p_n - A_{n-1} C_n /(q(q-1))^2 p_{n-1},
\end{equation}
pour les versions unitaires $p_n$ des polynômes $P_n$ définis par
\eqref{hypergeo_cd}.

Considérons les trois cas particuliers qui nous intéressent.

$\bullet$ Pour $(c,d)=(0,0)$, on obtient
\begin{equation}
  A_n = \frac{(1-q^{n+1})^3}{(1-q^{2n+1})(1-q^{2n+2})}
\end{equation}
et
\begin{equation}
  C_n = -\frac{q^{n+1}(1-q^n)^3}{(1-q^{2n})(1-q^{2n+1})}.
\end{equation}
La récurrence est donc donnée par
\begin{equation}
  % OK, c'est vérifié
  \label{recu00}
  p_{n+1} = \left(\frac{[2n+1]_q+[n+1]_q-3[n]_q}{(1+q^n)(1+q^{n+1})} + x\right) p_n + \frac{q^{n-1}[n]_q^6}{[2n-1]_q[2n]_q^2[2n+1]_q} p_{n-1}.
\end{equation}

\medskip

$\bullet$ Pour $(c,d)=(0,1)$, on obtient
\begin{equation}
  A_n = \frac{(1-q^{n+1})(1-q^{n+2})^2}{(1-q^{2n+2})(1-q^{2n+3})}
\end{equation}
et
\begin{equation}
  C_n = -\frac{q^{n+2}(1-q^n)^2(1-q^{n+1})}{(1-q^{2n+1})(1-q^{2n+2})}.
\end{equation}
La récurrence est donc donnée par
\begin{equation}
  \label{recu01}
  p_{n+1} = \left(a_n + x\right) p_n +  \frac{q^n [n]_q^3[n+1]_q^3}{[2n]_q[2n+1]_q^2[2n+2]_q} p_{n-1},
\end{equation}
où le coefficient $a_n = (A_n + C_n - 1 + q)/(q(q-1))$ ne se simplifie pas spécialement.

$\bullet$ Pour $(c,d)=(1,1)$, on obtient
\begin{equation}
  A_n = \frac{(1-q^{n+2})^2(1-q^{n+3})}{(1-q^{2n+3})(1-q^{2n+4})}
\end{equation}
et
\begin{equation}
  C_n = -\frac{q^{n+3}(1-q^n)(1-q^{n+1})^2}{(1-q^{2n+2})(1-q^{2n+3})}.
\end{equation}
La récurrence est donc donnée par
\begin{equation}
  \label{recu11}
  % ok vérifié
  p_{n+1} = \left(\frac{(q-1)[n+1]_q[n+2]_q}{(1+q^{n+1})(1+q^{n+2})} + x\right) p_n + \frac{q^{n+1}[n]_q[n+1]_q^4[n+2]_q}{[2n+1]_q[2n+2]_q^2[2n+3]_q} p_{n-1}.
\end{equation}

\subsection{Récurrence pour le type grand $q$-Legendre}

Selon \cite[\S 14.5.1]{koekoek_book}, la version unitaire $q_n$ des polynômes $Q_n$ définis
par la formule \eqref{hypergeo_avec_z} (avec $x$ au lieu de $q(1+(q-1)x)$) vérifie la récurrence
\begin{equation}
  q_{n+1} = (A_n + C_n - 1 + x) q_n - A_{n-1} C_n q_{n-1},
\end{equation}
où
\begin{equation}
  A_n = \frac{(1-q^{n+1})^2(1-(1+(q-1)z) q^{n+1})}{(1-q^{2n+1})(1-q^{2n+2})}
\end{equation}
et
\begin{equation}
  C_n = -\frac{q^{n+1}(1-q^n)^2(1-q^n+(q-1)z)}{(1-q^{2n})(1-q^{2n+1})}.
\end{equation}
On utilise ensuite comme précédemment le lemme \ref{affine} pour le changement de variables $x \mapsto q(1+(q-1)x)$.

La récurrence pour la version unitaire $p_n$ des $P_n$ définis par \eqref{hypergeo_avec_z} est donc donnée par
\begin{equation}
  \label{recu00z}
  p_{n+1} = \left(\frac{[2n+1]_q+[n+1]_q-3[n]_q-2 q^n z}{(1+q^n)(1+q^{n+1})} + x\right) p_n + \frac{q^{n-1}[n]_q^4([n]_q-z)([n]_q+q^n z)}{[2n-1]_q[2n]_q^2[2n+1]_q} p_{n-1}.
\end{equation}

\section{Déterminant de Hankel et fractions continues}

\label{hankel}

La théorie générale des polynômes orthogonaux donne des informations
précises sur les déterminants de Hankel des moments, et sur certaines
fractions continues pour la série génératrice ordinaire des
moments. Les résultats dont nous aurons besoin se trouvent dans
\cite[\S 2.7]{kratt_det} et \cite[\S 5.4]{kratt_det2}, où le
lecteur peut trouver d'autres références. On les rassemble dans le
théorème-omnibus suivant.

\begin{theorem}
  \label{omnibus}
  Soit $(p_n)_{n \geq 0}$ une famille de polynômes orthogonaux unitaires
  vérifiant la récurrence
  \begin{equation}
    \label{eq:pn}
    p_{n+1} = (a_n + x) p_n - b_n p_{n-1},
  \end{equation}
  avec les conditions initiales $p_{-1}=0$ et $p_0=1$. Soient $\mu_n$
  les moments de cette famille, qu'on normalise en supposant $\mu_0=1$. Alors on a un développement en fraction continue
  \begin{equation}
    \label{Jfraction}
    \sum_{k \geq 0} \mu_k x^k = \cfrac{1}{1+ a_0 x
      - \cfrac{b_1 x^2}{1+ a_1 x
        - \cfrac{b_2 x^2}{1+ a_2 x - \dots}}}
  \end{equation}
  et une factorisation du déterminant de Hankel
  \begin{equation}
    \label{eqgeneral1}
    d_{n}^{(0)} := \det_{0 \leq i,j \leq n-1} \mu_{i+j} = \prod_{k=1}^{n-1}b_k^{n-k}.
  \end{equation}

  Si $(q_n)_{n\geq 0}$ est la suite définie par la récurrence
  \begin{equation}\label{eq:qn}
    q_0 = 1,\quad q_1 = - a_0 \quad \text{et}\quad q_{n+1} = - a_n q_n - b_n q_{n-1},
  \end{equation}
  alors on a aussi une factorisation du déterminant de Hankel décalé
  \begin{align}
    d_{n}^{(1)} := \det_{0\leq i,j\leq n-1} \mu_{i+j+1}&=q_n \prod_{k=1}^{n-1}b_k^{n-k}.\label{eqgeneral2}
  \end{align}

\end{theorem}
\begin{proof}
  On renvoie aux références citées pour la preuve des principaux
  résultats énoncés. On se contente ici d'une esquisse de preuve de
  \eqref{eqgeneral2}.

  En comparant \eqref{eq:qn} et \eqref{eq:pn}, on voit que
  $q_n=(-1)^np_n(0)$. Par ailleurs, il existe une formule déterminantale pour  $p_n(x)$ en fonction des moments :
  \begin{align}\label{eqexplicit}
    p_n(x)=\frac{1}{d_{n}^{(0)}}\left|
      \begin{array}{cccc}
        \mu_0 & \mu_1 & \ldots & \mu_{n} \\
        \mu_1 & \mu_2 & \ldots & \mu_{n+1} \\
        \vdots & \vdots & \vdots & \vdots  \\
        \mu_{n-1} & \mu_{n} & \ldots & \mu_{2n-1} \\
        1 & x  & \ldots & x^{n} \\
      \end{array}
    \right|.
  \end{align}
  En posant $x=0$ dans \eqref{eqexplicit}, on obtient
  $p_n(0)=(-1)^n d_{n}^{(1)}/d_{n}^{(0)}$, ce qui équivaut à \eqref{eqgeneral2}. 
\end{proof}

On obtient ainsi des fractions continues pour $\widehat{B}(x)$,
$\widehat{B}_1(x)$, $\widehat{B}_2(x)$ et $\widehat{B}_z(x)$, de la
forme donnée par \eqref{Jfraction} dans le théorème \ref{omnibus},
ayant pour coefficients ceux des récurrences \eqref{recu00},
\eqref{recu01}, \eqref{recu11} et \eqref{recu00z}.

On obtient aussi les formules suivantes de factorisation des déterminant de Hankel.
\begin{theorem}
  On a
  \begin{align}
    \det_{0\leq i,j\leq n-1}\beta_{i+j}&=(-1)^{\binom{n}{2}} q^{\binom{n}{3}}\prod_{i=1}^{n-1}
    \frac{[i]!_q^6}{[2i]!_q[2i+1]!_q},\label{hank0}\\
\det_{0\leq i,j\leq n-1}\beta_{i+j+1}&=\frac{(-1)^{\binom{n+1}{2}}}{[2]_q} q^{\binom{n+1}{3} }\prod_{i=1}^{n-1}
\frac{[i]!_q^3 [i+1]!_q^3}{[2i+1]!_q[2i+2]!_q},\label{hank1}\\
\det_{0\leq i,j\leq n-1}\beta_{i+j+2}&=
\frac{(-1)^{\binom{n}{2}}}{[2]_q[3]_q} q^{\binom{n+2}{3}}
\prod_{i=1}^{n-1}\frac{[i]!_q [i+1]!_q^4[i+2]!_q}{[2i+2]!_q[2i+3]!_q},\label{hank2}\\
\det_{0\leq i,j\leq n-1}\beta_{i+j+3}&=
\frac{(-1)^{\binom{n+1}{2}}}{[3]_q^2 [4]_q} q^{\binom{n+2}{3}} \left(q^{\binom{n+2}{2}}+(-1)^n\right) \prod_{i=1}^{n-1} \frac{[i+1]!_q^3 [i+2]!_q^3}{[2i+3]!_q [2i+4]!_q}\label{hank3}
\end{align}
et
\begin{equation}
  \label{hankz}
  \det_{0 \leq i,j \leq n-1} \frac{1}{z}\int_{0}^{z}\beta_{i+j}(y) dy = (-1)^{\binom{n}{2}} q^{\binom{n}{3}} \prod_{i=1}^{n-1} \frac{[i]!_q^4 \asc(z,i)\desc(z,i)}{[2i]!_q [2i+1]!_q},
\end{equation}
avec les notations \eqref{def_asc} et \eqref{def_desc}.
\end{theorem}
\begin{proof}
  Les formules \eqref{hank0}, \eqref{hank1}, \eqref{hank2} et
  \eqref{hankz} s'obtiennent directement en appliquant \eqref{eqgeneral1} aux
  quatre récurrences obtenues dans la section \ref{recurs}. Il faut
  tenir compte de la normalisation des moments pour \eqref{hank1} et
  \eqref{hank2}.

  On pose
  \begin{equation*}
    d_n(k)=\det_{0\leq i,j\leq n-1}\beta_{i+j+k}.
  \end{equation*}
  On utilise \eqref{eqgeneral2} pour montrer successivement les trois
  implications suivantes.
  \begin{itemize}
  \item[\eqref{hank0} $\Rightarrow$ \eqref{hank1}]
    Les polynômes orthogonaux unitaires associés aux $(\beta_n)_{n \geq 0}$ sont
    \begin{equation*}
      p_n(x)=\frac{(q;q)_n^2}{(q^{n+1}; q)_n}\frac{1}{q^{n}(q-1)^{n}}
      \sum_{k=0}^n\frac{(q^{-n},q^{n+1}; q)_k(q(1+(q-1)x);q)_kq^k}{(q,q,q;q)_k}.
    \end{equation*}
    Par conséquent, à l'aide de l'identité de $q$-Vandermonde,
    \begin{equation*}
      p_n(0)=\frac{(q;q)_n^2}{(q^{n+1}; q)_n} \frac{q^{\binom{n}{2}}}{(1-q)^n}=q^{\binom{n}{2}} \frac{[n]!_q^3 }{[2n]!_q}.
    \end{equation*}
    Il en résulte que $d_n(1)=d_n(0)(-1)^n p_n(0)$, ce qui donne une autre preuve de \eqref{hank1}.
  \item[ \eqref{hank1} $\Rightarrow$ \eqref{hank2}]
    Les polynômes orthogonaux unitaires associés aux $(\beta_{n+1}/\beta_1)_{n\geq 0}$ sont
    \begin{equation*}
      p_n(x)=\frac{(q,q^2;q)_n}{(q^{n+2}; q)_n}\frac{1}{q^{n}(q-1)^{n}}
      \sum_{k=0}^n\frac{(q^{-n},q^{n+2}; q)_k(q(1+(q-1)x);q)_kq^k}{(q,q,q^2;q)_k}.
    \end{equation*}
    Par conséquent, via $q$-Vandermonde, 
    \begin{equation*}
      p_n(0)=\frac{(q;q)_n^2}{(q^{n+2}; q)_n} \frac{q^{\binom{n+1}{2}}}{(1-q)^n}=q^{\binom{n+1}{2}} \frac{[n]!_q^2 [n+1]!_q}{[2n+1]!_q}.
    \end{equation*}
    Donc $d_n(2)=d_n(1)(-1)^n p_n(0)$, ce qui donne une autre preuve de \eqref{hank2}.
  \item[ \eqref{hank2} $\Rightarrow$ \eqref{hank3}]
    Les polynômes orthogonaux unitaires associés aux $(\beta_{n+2}/\beta_2)_{n\geq 0}$ sont
    \begin{equation*}
      p_n(x)=\frac{(q^2,q^2;q)_n}{(q^{n+3}; q)_n}\frac{1}{q^{n}(q-1)^{n}}
      \sum_{k=0}^n\frac{(q^{-n},q^{n+3}; q)_k(q(1+(q-1)x);q)_kq^k}{(q,q^2,q^2;q)_k}.
    \end{equation*}
    dont la valeur en $x=0$ est
    \begin{equation}
      p_n(0)=\frac{(q^2,q^2;q)_n}{(q^{n+3}; q)_n}\frac{1}{q^{n}(q-1)^{n}}
      \sum_{k=0}^n\frac{(q^{-n},q^{n+3}; q)_kq^k}{(q^2,q^2;q)_k}.
    \end{equation}
    La somme qui intervient peut se simplifier comme suit (par $q$-Vandermonde):
    \begin{align*}
      \sum_{k=0}^n\frac{(q^{-n},q^{n+3}; q)_kq^k}{(q^2,q^2;q)_k}&=
      \frac{(1-q)^2q^{-1}}{(1-q^{-n-1})(1-q^{n+2})}\left(-1+\pFq{2}{1}{q^{-n-1},q^{n+2}}{q}\right)\\
      &=(-1)^n\frac{q^n(1-q)^2}{(1-q^{n+1})(1-q^{n+2})}\left[(-1)^n +q^{\binom{n+2}{2}}\right].
    \end{align*}
    On obtient donc l'égalité
    \begin{equation*}
      p_n(0)=\frac{(q^2,q^2;q)_n}{(q^{n+1}; q)_{n+2}}\frac{1}{(1-q)^{n-2}}
      \left[(-1)^n +q^{\binom{n+2}{2}}\right]= \frac{[n]!_q [n+1]!_q^2}{[2n+2]!_q}\left[(-1)^n +q^{\binom{n+2}{2}}\right].
    \end{equation*}
    Il s'ensuit que $d_n(3)=d_n(2)(-1)^n p_n(0)$, ce qui donne \eqref{hank3}.
  \end{itemize}
\end{proof}

L'expression \eqref{hank3} est à rapprocher de la fraction continue
simple pour la série $\widehat{B}_2$ obtenue dans la section
\ref{autres_f}, qui fait intervenir des facteurs similaires.

Pour les décalages $D \geq 4$, le nombre $\beta_{D}$ lui-même a des
racines en dehors du cercle unité, donc il est impossible que les
déterminants de Hankel soient encore des produits de polynômes
cyclotomiques.

% Dans le cas classique, les décalages 4 et 5 semblent reliés aux suites
% A093158 et A093159 de l'encyclopédie de Sloane. Les matrices se découpent
% en deux blocs à cause des annulations des $B_n$.

\section{Autres fractions continues}
\label{autres_f}

On obtient dans cette section d'autres fractions continues pour les
mêmes séries génératrices. Les fractions continues du type donné par
\eqref{Jfraction} sont traditionnellement nommées des J-fractions
continues ou fractions continues de Jacobi. On les transforme ici en des
fractions continues de Stieltjes ou S-fractions continues.

On a besoin du lemme de transformation suivant (voir \cite[Lemma I]{rogers}, \cite[Lemme 5.3]{chen_kw} et \cite{dumont}), qui permet de relier S-fractions continues et J-fractions continues.
\begin{lemma}
  \label{transfo_fc}
  On a l'égalité entre les deux développements en fractions continues
  \begin{equation}
    \frac{1}{1+}\,\frac{c_1 x}{1+}\,\frac{c_2 x}{1+}\,\frac{c_3 x}{1+}\dots
  =
    \frac{1}{1+c_1 x-}\,\frac{c_1 c_2 x^2}{1+(c_2+c_3)x -}\,
    \frac{c_3 c_4 x^2}{1+(c_4+c_5)x -}\dots
  \end{equation}
\end{lemma}

\begin{theorem}
  On a le développement en fraction continue
  \begin{equation}
    \widehat{B}(x) =  \cfrac{1}{[1]_q
      + \cfrac{x}{\frac{q+1}{[1]_q}
        - \cfrac{x}{[3]_q
          + \cfrac{q[2]_q x}{\frac{q^2+1}{[2]_q}
            - \cfrac{[2]_q x}{[5]_q 
              + \cfrac{q^2[3]_q x}{\frac{q^3+1}{[3]_q}
                - \cfrac{[3]_q x}{\ddots
                }}}}}}}
  \end{equation}
\end{theorem}
\begin{proof}
  En termes du lemme \ref{transfo_fc}, ce développement revient à montrer que
  \begin{equation}
    c_{2n-1} = \frac{q^{n-1}[n]_q^2}{(q^n+1)[2n-1]_q} ,\quad\text{et}\quad
    c_{2n}= -\frac{[n]_q^2}{(q^n+1)[2n+1]_q} ,
  \end{equation}
  pour $n \geq 1$. Il reste donc à vérifier par un simple calcul que
  \begin{equation}
    a_0  = c_1, \quad %ok
    a_n  = c_{2n} + c_{2n+1}, \quad %ok
    b_n  = c_{2n-1} c_{2n},  %ok
  \end{equation}
  pour les coefficients $a_n$ et $b_n$ de la récurrence
  \eqref{recu00}.
\end{proof}

\begin{theorem}
  On a le développement en fraction continue
  \begin{equation}
    \widehat{B}_1(x) =  \cfrac{1}{1
      + \cfrac{\frac{q}{[3]_q}x}{1
        - \cfrac{\frac{[2]_q^2}{[3]_q[4]_q}x}{1
          + \cfrac{\frac{q^2 [2]^2_q[3]_q}{[4]_q[5]_q} x}{\ddots
          }}}}
  \end{equation}
  dont les coefficients alternent entre $
    \frac{q^n[n]_q^2[n+1]_q}{[2n]_q[2n+1]_q}$ et
$    \frac{-[n]_q[n+1]_q^2}{[2n+1]_q[2n+2]_q}$.
\end{theorem}
\begin{proof}
  En termes du lemme \ref{transfo_fc}, ce développement revient à montrer que
  \begin{equation}
    c_{2n -1} =  \frac{q^n[n]_q^2[n+1]_q}{[2n]_q[2n+1]_q},\quad\text{et}\quad
    c_{2n} = -\frac{[n]_q[n+1]_q^2}{[2n+1]_q[2n+2]_q},
  \end{equation}
  pour $n \geq 1$. Il reste donc à vérifier par un simple calcul que
 \begin{equation}
    a_0  = c_1, \quad %ok
    a_n  = c_{2n} + c_{2n+1}, \quad %ok
    b_n  = c_{2n-1} c_{2n},  %ok
  \end{equation}
  pour les coefficients $a_n$ et $b_n$ de la récurrence
  \eqref{recu01}.
\end{proof}

On peut en déduire facilement un développement du même type pour la fonction
$1/\widehat{B}$, en utilisant la relation $\widehat{B}=1+\beta_1 \widehat{B}_1$.

Une formule de ce type existe aussi pour la
série $\widehat{B}_2$ avec un décalage de deux crans. 

\begin{theorem}
  On a le développement en fraction continue
  \begin{equation}
    \widehat{B}_2(x) =  \cfrac{1}{1
      + \cfrac{c_1 x}{1
        + \cfrac{c_2 x}{1
          + \cfrac{c_3 x}{\ddots
          }}}}
  \end{equation}
  dont les coefficients alternent entre 
  \begin{equation}
    c_{2n-1} = \frac{ [n]_q [n + 1]_q^2}{[2n + 1]_q [2n + 2]_q}
    \frac{(q ^ {\binom{n + 2}{2}} + (-1) ^ {n + 2})}
    {(q ^ {\binom{n + 1}{2}} + (-1) ^ {n + 1})}
  \end{equation}
  et
  \begin{equation}
    c_{2n}=
    \frac{- q^{n + 1} [n + 1]_q^2 [n + 2]_q}{[2n + 2]_q [2n + 3]_q}
    \frac{(q^{\binom{n + 1}{2}} + (-1)^{n + 1})}
    {(q^{\binom{n + 2}{2}} + (-1) ^ {n + 2})},
  \end{equation}
  pour $n\geq 1$.
\end{theorem}
\begin{proof}
  Il suffit de vérifier par un simple calcul que
  \begin{equation}
    a_0  = c_1 = \frac{q-1}{q^2+1}, \quad %ok
    a_n  = c_{2n} + c_{2n+1}, \quad %TODO
    b_n  = c_{2n-1} c_{2n},  %ok
  \end{equation}
  pour les coefficients $a_n$ et $b_n$ de la récurrence
  \eqref{recu11}.
\end{proof}

On observe qu'il apparaît des facteurs cyclotomiques d'ordre $n(n+1)$
ou $n(n+1)/2$. Cette expression est à rapprocher du comportement des
déterminants de Hankel pour le décalage de $3$, voir section
\ref{hankel}. Par ailleurs, les coefficients ont des pôles en $q=1$,
de sorte que ce développement n'a pas de version classique.

\begin{theorem}
  On a le développement en fraction continue
  \begin{equation}
    \widehat{B}_z(x) =  \cfrac{1}{[1]_q
      + \cfrac{([1]_q-z) x}{\frac{q+1}{[1]_q}
        - \cfrac{([1]+q z)x}{[3]_q
          + \cfrac{q([2]_q-z) x}{\frac{q^2+1}{[2]_q}
            - \cfrac{([2]_q+q^2 z) x}{[5]_q 
              + \cfrac{q^2([3]_q -z) x}{\frac{q^3+1}{[3]_q}
                - \cfrac{([3]_q +q^3 z)x}{\ddots
                }}}}}}}
  \end{equation}
\end{theorem}
\begin{proof}
  En termes du lemme \ref{transfo_fc}, ce développement revient à montrer que
  \begin{equation}
    c_{2n-1} = \frac{q^{n-1}[n]_q([n]_q-z)}{(q^n+1)[2n-1]_q} ,\quad\text{et}\quad
    c_{2n}= -\frac{[n]_q([n_q+q^n z)}{(q^n+1)[2n+1]_q} ,
  \end{equation}
  pour $n \geq 1$. Il reste donc à vérifier par un simple calcul que
  \begin{equation}
    a_0  = c_1, \quad 
    a_n  = c_{2n} + c_{2n+1}, \quad 
    b_n  = c_{2n-1} c_{2n},  
  \end{equation}
  pour les coefficients $a_n$ et $b_n$ de la récurrence
  \eqref{recu00z}.
\end{proof}

% \section{Propriétés de symétrie}

% VIRER TOUT CA ?

% $\bullet$ $q$-symétrie des polynômes orthogonaux

% pour les $P_{n,0}$:

% Par exemple $P_n(-q(1+q x),1/q) = (-1)^{n} q^{2n} P_n(x,q)$ 

% (pas de symétrie pour les $P_{n,1}$)

% pour les $P_{n,2}$:

% soit $P_n(-q x,1/q) = (-1)^{n} q^{?} P_n(x,q)$

% même genre de symétrie que les polynômes d'Ehrhart des polytopes Gorenstein

\section{Théorème général de type Fulmek-Krattenthaler}

On considère les polynômes en $x$ définis par la formule
\begin{equation}
  \label{hypergeo_abcd}
  P_n(x) = \pFq{3}{2}{q^{-n}, q^{n + a + b + c + d - 1},
    q^a (1 + (q - 1) x)}{q^{a + c}, q^{a + d}}
\end{equation}
où ${}_3\phi_{2}$ est la fonction hypergéométrique basique
usuelle. Les paramètres $c$ et $d$ sont des entiers positifs ou
nuls. Les paramètres $a$ et $b$ sont des entiers strictement positifs.

On retrouve les polynômes de type $q$-Hahn considérés précédemment lorsque $a=b=1$.

Ces polynômes sont l'évaluation en $q^a (1 + (q - 1) x)$ de polynômes
$Q_n$ de type ``grand $q$-Jacobi'' pour les paramètres $(q^{a+c-1},
q^{b+d-1}, q^{a+d-1})$. Ce sont encore des polynômes orthogonaux (voir
le lemme \ref{affine}).

On pose $C_{a,b,c,d} = \Psi\left(x^{2} \asc(x, a-1) \asc(x, b-1)
  \desc(x, c-1) \desc(x, d-1)\right)$.
% possible d'evaluer, mais est-ce utile ?

\begin{theorem}
  Les nombres 
  \begin{equation}
    \Psi\left(x^{n+2} \asc(x, a-1) \asc(x, b-1) \desc(x, c-1) \desc(x, d-1)\right) / C_{a,b,c,d}
  \end{equation}
  sont les
  moments des polynômes orthogonaux $P_n$ de paramètres $a,b,c,d$.
\end{theorem}
Ce résultat est un $q$-analogue du théorème 23 de \cite{fulmek-kratt}. On peut noter la remarquable symétrie par échange de $a$ et $b$ ou de $c$ et $d$, qui n'est pas immédiatement visible dans \eqref{hypergeo_abcd}.

\begin{proof}
  Il suffit à nouveau de montrer que la forme linéaire
  \begin{equation}
    f \mapsto
    \Psi\left(x^2 \asc(x, a-1) \asc(x, b-1) \desc(x, c-1) \desc(x, d-1) f\right)
  \end{equation}
  s'annule sur les polynômes $P_n$ lorsque $n\geq 1$.

  L'expression hypergéometrique \eqref{hypergeo_abcd} donne la formule explicite
  \begin{equation}
    P_n(x) = \sum_{k=0}^{n} \frac{(q^{-n}, q^{n+a+b+c+d-1} ; q)_k (q^a (1 + (q - 1) x);q)_k q^k}{(q,q^{a+c},q^{a+d};q)_k}.
  \end{equation}

  En considérant l'expression
  \begin{equation}
    \label{xxabcd}
    x^2  \asc(x, a-1) \asc(x, b-1) \desc(x, c-1) \desc(x, d-1) (q^a (1 + (q - 1) x);q)_k,
\end{equation}
 on observe qu'on peut associer ensemble les facteurs pour former d'une part
  \begin{equation}
    \desc(x, c-1) x \asc(x, a-1) (q^a (1 + (q - 1) x);q)_k = (1-q)^k (-1)^{c-1} q^{\binom{c}{2}} [a+k-1+c]!_q \qbase{a+k-1}{x}{a+k-1+c}_q
  \end{equation}
  et d'autre part
  \begin{equation}
    \desc(x, d-1) x \asc(x, b-1) = (-1)^{d-1} q^{\binom{d}{2}} [b+d-1]!_q \qbase{b-1}{x}{b+d-1}_q.
  \end{equation}
  Au complet, l'expression \eqref{xxabcd} vaut donc
  \begin{equation}
    (q-1)^k (-1)^{c+d} q^{\binom{c}{2}+\binom{d}{2}} [a+k-1+c]!_q [b+d-1]!_q \qbase{a+k-1}{x}{a+k-1+c}_q \qbase{b-1}{x}{b+d-1}_q.
  \end{equation}

  Le lemme \ref{alpha} donne que
  \begin{equation}
    \label{eval_psi_gen}
    \Psi(\qbase{a+k-1}{x}{a+k-1+c}_q \qbase{b-1}{x}{b+d-1}_q)= \frac{(-1)^{c+d}q^{-\binom{c}{2}+c d-\binom{d}{2}}}{[a+b+c+d+k-1]_q \qbinom{a+b+c+d+k-2}{b+c-1}_q}.
  \end{equation}
  On en déduit alors que
  \begin{multline}
    \Psi(x^2  \asc(x, a-1) \asc(x, b-1) \desc(x, c-1) \desc(x, d-1) P_n(x)) = \\
 q^{cd} [b+d-1]!_q[b+c-1]!_q \sum_{k=0}^{n} \frac{(q^{-n}, q^{n+a+b+c+d-1} ; q)_k q^k (1-q)^k  [a+c+k-1]!_q[a+d+k-1]!_q}{(q,q^{a+c},q^{a+d};q)_k [a+b+c+d+k-1]!_q} \\=
 q^{cd} \frac{[b+d-1]!_q[b+c-1]!_q [a+c -1]!_q [a+d -1]!_q}{[a+b+c+d-1]!_q}\sum_{k=0}^{n} \frac{(q^{-n}, q^{n+a+b+c+d-1} ; q)_k q^k (q^{a+c}; q)_k (q^{a+d}; q)_k}{(q,q^{a+c},q^{a+d};q)_k (q^{a+b+c+d};q)_k} .
  \end{multline}
  On reconnaît la somme comme $\pFq{2}{1}{q^{-n}, q^{n + a+b+c+d -1 }}{q^{a+b+c+d}}$, qui
  est nul par la formule de $q$-Vandermonde lorsque $n\geq 1$.
\end{proof}

Comme sous-produit de cette preuve, on obtient une expression pour la
constante de normalisation $C_{a,b,c,d}$ :
\begin{equation}
   C_{a,b,c,d} = q^{cd} \frac{[b+d-1]!_q[b+c-1]!_q [a+c -1]!_q [a+d -1]!_q}{[a+b+c+d-1]!_q}.
\end{equation}

On peut alors en déduire un énoncé sur la factorisation des
déterminants de Hankel.

Selon \cite[\S 14.5]{koekoek_book}, la version unitaire $q_n$ des polynômes
$Q_n$ définis par la formule \eqref{hypergeo_abcd} (avec $x$ à la place de $q^a(1 + (q - 1) x)$) vérifie la récurrence
\begin{equation}
  q_{n+1} = (A_n + C_n - 1 + x) q_n - A_{n-1} C_n q_{n-1},
\end{equation}
où
\begin{equation}
  A_n = \frac{(1-q^{n+a+d})(1-q^{n+a+c})(1-q^{n+a+b+c+d-1})}{(1-q^{2n+a+b+c+d-1})(1-q^{2n+a+b+c+d})}
\end{equation}
et
\begin{equation}
  C_n = -\frac{q^{n+2a+c+d-1}(1-q^n)(1-q^{n+b+d-1})(1-q^{n+b+c-1})}{(1-q^{2n+a+b+c+d-2})(1-q^{2n+a+b+c+d-1})}.
\end{equation}

On utilise ensuite le lemme \ref{affine} pour le changement de
variables $x \mapsto q^a(1+(q-1)x)$. Dans la récurrence obtenue pour
les versions unitaires $p_n$ des polynômes $P_n$ définis par
\eqref{hypergeo_abcd}, les coefficients $b_n$ sont donc
\begin{equation}
   - q^{n+c+d-1}  \frac{[n]_q [a + c + n - 1]_q [b + c + n - 1]_q
               [a + d + n - 1]_q [b + d + n - 1]_q
               [a + b + c + d + n - 2]_q }{
               [a + b + c + d + 2 n - 3]_q
                [a + b + c + d + 2 n - 2]_q ^ 2
                [a + b + c + d + 2 n - 1]_q }.
\end{equation}

Soit $M(n)$ la matrice de terme général
\begin{equation}
  \Psi(x^{i + j + 2} \asc(x, a-1) \asc(x, b-1) \desc(x, c-1) \desc(x, d-1))
\end{equation}
pour $0 \leq i,j \leq n-1$.

On déduit donc du théorème \ref{omnibus} le résultat suivant.
\begin{theorem}
  Le déterminant de $M(n)$ est 
  \begin{multline}
    (-q^{c+d})^{\binom{n}{2}} q^{\binom{n}{3}} C_{a,b,c,d} ^ n \times \\
     \prod_{i=1}^{n-1} \left( \frac{[i]_q [a + c + i - 1]_q [b + c + i - 1]_q
               [a + d + i - 1]_q [b + d + i - 1]_q
               [a + b + c + d + i - 2]_q }{
               ([a + b + c + d + 2 i - 3]_q
                [a + b + c + d + 2 i - 2]_q ^ 2
                [a + b + c + d + 2 i - 1]_q) }\right)^ {n - i}.
  \end{multline}
\end{theorem}

On retrouve les déterminants de Hankel décalés des $\beta_n$ pour
$a=b=1$ et $(c,d)=(0,0),(0,1),(1,1)$ (formules \eqref{hank0}, \eqref{hank1} et \eqref{hank2}).  %ok, ca marche

% \section{Remarques diverses}

% \subsection{Polygone de Newton}

% On peut noter la forme intriguante du polygone de Newton du numérateur
% des $P_{n,a=1,b=1,c=0,d=0}$ (qui n'est pas positif).

% \begin{figure}[ht]
%   \begin{center}
%     \includegraphics[width=15cm]{numerator_ortho_poly_11.pdf}
%     \caption{Support et signes des coefficients du numérateur du polynôme orthogonal $P_{11}$ pour $(a,b,c,d)=(1,1,0,0)$}
%     \label{support}
%   \end{center}
% \end{figure}

% On voit à gauche une zone où les coefficients sont de signes qui
% alternent, une zone au milieu ou se passent des fusion entre bandes,
% puis une zone ou les signes sont constants. Il semble qu'il existe une
% forme asymptotique pour les séparations entre zones.

\bibliographystyle{alpha}
\bibliography{hankel_bernoulli}

{Fr\'ed\'eric Chapoton} \\
{Institut Camille Jordan, CNRS UMR 5208, Université Claude Bernard Lyon 1,
 Université de Lyon,
    F-69622 Villeurbanne cedex, France}            \\
{chapoton@math.univ-lyon1.fr}

{Jiang Zeng} \\
{Institut Camille Jordan, CNRS UMR 5208, Université Claude Bernard Lyon 1,
 Université de Lyon, 
    F-69622 Villeurbanne cedex, France}            \\         
{zeng@math.univ-lyon1.fr}

\end{document}